\documentclass[11pt]{article}
\usepackage{amsmath,amsthm,amsfonts,amssymb,mathrsfs,epsfig,bm,latexsym}
\usepackage[usenames]{color}
\usepackage[colorlinks=true]{hyperref}
\usepackage{xspace}
\newtheorem{theorem}{Theorem}%[section]
\newtheorem{lemma}{Lemma}
\newtheorem{corollary}{Corollary}
\newtheorem{proposition}{Proposition}

\newtheorem{remark}{Remark}

\newtheorem{definition}{Definition}
\long\def\symbolfootnote[#1]#2{\begingroup
\def\thefootnote{\fnsymbol{footnote}}\footnote[#1]{#2}\endgroup}
\def\N{\mathbb{N}}
\def\Z{\mathbb{Z}}

\def\R{\mathbb{R}}
\def\C{\mathbb{C}}
\def\P{\mathbb{P}}
\def\E{\mathbb{E}}

\renewcommand{\phi}{\varphi}
\renewcommand{\epsilon}{\varepsilon}
\newcommand{\1}{{\text{\Large $\mathfrak 1$}}}
\newcommand{\comp}{\raisebox{0.1ex}{\scriptsize $\circ$}}

\definecolor{mygray}{gray}{0.9}

\newcommand{\keywords}[1]{ \noindent {\footnotesize
             {\small \em Keywords and phrases.} {\sc #1} } }
\newcommand{\ams}[2]{  \noindent {\footnotesize
             {\small \em AMS {\rm 2010} subject classifications.
             {\rm Primary {\sc #1}; secondary {\sc #2}} } } }

\def\heads{\mathsf H}
\def\tails{\mathsf T}
\def\explo{\mathsf E}
\def\geo{\operatorname{geo}}

\begin{document}
\title{\Large \bf Runs in coin tossing: a general approach for deriving 
distributions for functionals}
\author{\sc Lars Holst \and \sc Takis Konstantopoulos\thanks{Corresponding author; Supported by Swedish Research Council grant 2013-4688}}
\date{12 July 2014\footnote{Slightly modified version from
the earlier on written on 10 February 2014}}
\maketitle

\begin{abstract}
We take a fresh look at the classical problem of runs in a sequence
of i.i.d.\ coin tosses and derive a general identity/recursion which
can be used to compute (joint) distributions of functionals of
run types. This generalizes and unifies already existing approaches.
We give several examples, derive asymptotics, and pose some further
questions.

\vspace*{2mm}
\keywords{regeneration, coin tossing, runs, longest run, Poisson approximation,
Laplace transforms, Rouch\'e's theorem}

\vspace*{2mm}
\ams{60G40}{60G50, 60F99, 60E10}
%60G40   	Stopping times; optimal stopping problems; gambling theory 
%60G50   	Sums of independent random variables; random walks
%60F99   	None of the above, but in this section
%60E10   	Characteristic functions; other transforms

\end{abstract}

%\comments{Incomplete draft of working paper. Please do not distribute.}

\section{Introduction}
The tendency of ``randomly occurring events'' to clump together is a 
well-understood chance phenomenon which has occupied people since the birth
of probability theory. In tossing i.i.d.\ coins, we will, from time to time,
see ``long'' stretches of heads. The phenomenon has been studied and
quantified extensively.
For a bare-hands approach 
see Erd\H{o}s and R\'enyi \cite{ER70} its sequel paper by Erd\H{o}s and 
R\'ev\'esz \cite{ER76} and the review paper by R\'ev\'esz \cite{R78}.

We shall consider a sequence $(\xi_n, n\in \N)$ of Bernoulli random variables 
with $\P(\xi_n=1)=p$, $\P(\xi_n=0)=q=1-p$, and let
\[
S(n) := \xi_1+\cdots+\xi_n, \quad n \ge 1, \quad S(0):=0.
\]
Throughout the paper, 
a ``run'' refers to an interval $I \subset \N:=\{1,2,\ldots\}$ 
such that $\xi_n=1$ for all $n\in I$ and
there is no interval $J \supset I$ such that $\xi_n=1$ for all $n\in J$.
There has been an interest in computing the distribution of runs of 
various types such as the number 
of runs of a given  length in $n$ coin tosses. 
Feller \cite[Section XIII.7]{FEL68} considers the probability that a run
of a given length $\ell$ first appears at the $n$-th coin toss
and, using renewal theory, computes the distribution of the number
of runs of a given length \cite[Problem 26, Section XIII.12]{FEL68}
as well as asymptotics \cite[Problem 25, Section XIII.12]{FEL68}.
(Warning: his definition of a run is slightly different.) He attributes
this result to 
von Mises \cite{VM1921}.\footnote{In his classic work \cite[p,\ 138]{VM},
von Mises, refers to a 1916 paper of the philosopher Karl Marbe who
reports that in 200,000 birth registrations in a town in Bavaria,
there is only one `run' of 17 consecutive births of children of
the same sex. Note that $\log_2(200000)\approx 17.61$ and see Section
\ref{lrpa} below.}
Philippou and Makri \cite{PM86} derive the joint distribution of 
the longest run and the number of runs of a given length.
More detailed computations are considered in \cite{MP11}.
The literature is extensive and there are two books on the topic
\cite{BK02,FL03}.

In this paper, we take a more broad view: we study real-
or vector-valued functionals of 
runs of various types and derive, using elementary methods,
an equation which can be specified at will to result into a formula
for the quantity of interest. To be more specific, let $R_\ell(n)$ be
the number of runs of length $\ell$ in the first $n$ coin tosses.
Consider the vector
\[
R(n) := (R_1(n), R_2(n), \ldots)
\]
as an element of the set
\[
\Z_+^* := \{x \in \Z_+^\N:\, x_k=0 \text{ eventually}\}
\]
which be identified with the set
$\bigcup_{\ell=1}^\infty \Z_+^\ell$ of nonempty words
from the alphabet of nonnegative integers, but, for the purpose of our
analysis, it is preferable to append, to each finite word, an infinite
sequence of zeros. This set is  countable, and so the
random variable $R(n)$ has a discrete distribution. If $h : \Z_+^* \to \R^d$
is any function then we refer to the random variable $h(R(n))$ as a 
$d$-dimensional functional of a run-vector. 
For example, for $d=1$, if $h_1(x) = \sup\{\ell:\,
x_\ell > 0\}$ (with $\sup \varnothing =0$), then $h_1(R(n))$
is the length of the longest run of heads in $n$ coin tosses.
If $h_2(x) := \sum_{\ell=1}^\infty \1\{x_\ell >0\}$, then $h_2(R(n))$ is
the total number of runs of any length in $n$ coin tosses.
Letting $d=2$, we may consider $h(x) := (h_1(x), h_2(x))$ as a 2-dimensional
functional; a formula for the distribution of $h(R(n))$ would then
be a formula for the joint distribution of the number of runs of a
given length together with the size of the longest run.
It is useful to 
keep in mind that $\Z^*:= \{x \in \Z^N:\, x_k=0, \text{ eventually}\}$,
is a vector space and that $\Z^*_+$ is a cone in this vector space.
If $x, y \in \Z^*$ then $x+y$ is defined component-wise.
The symbol $0$ denotes the origin $(0,0,\ldots)$ of this vector space.
For $j=1,2,\ldots$, we let  $e_j = (e_j(1), e_j(2), \ldots) \in \Z^*$ be
defined by
\[
e_j(n) := \1\{n=j\}, \quad n \in \N.
\]
It is convenient, and logically compatible with the last display, to set
\[
e_0 := (0,0,\ldots),
\]
thus having two symbols for the origin of the vector space $\Z^*$.

The paper is organized as follows. Theorem \ref{RECw} in Section \ref{portmid}
is a general formula for functionals of $R^*$, defined as
$R(\cdot)$ stopped at an independent geometric time.
We call this formula a ``portmanteau identity'' because it
contains lots of special cases of interest.
To explain this, we give, in the same section, formulas for
specific functionals. In Section \ref{exceed} we compute binomial
moments and distribution of $G_\ell(n) := \sum_{k \ge \ell} R_k(n)$,
the number of runs of length at least $\ell$ in $n$ coin tosses.
In particular, we point out its relationship with 
hypergeometric functions.
Section \ref{portmrec} translates the portmanteau identity
into a ``portmanteau recursion'' which provides, for example,
a method for recursive evaluation of the generating 
function of the random vector $R(n)$.
In Section \ref{lrpa} we take a closer look at the most common
functional of $R(n)$, namely the length $L(n)$ of the longest run
in $n$ coin tosses. We discuss the behavior of its distribution 
function and its relation to a Poisson approximation theorem, given
in Proposition \ref{PoiApp}
stating that
%. This roughly
%states that, in a certain approximating regime,
%the random variables 
$R_{\ell_1}(n), \ldots, R_{\ell_\nu}(n)$
become asymptotically independent Poisson random variables, as $n \to
\infty$, when, simultaneously, $\ell_1, \ldots, \ell_\nu \to \infty$.
A second approximation for the distribution function 
$\P(L(n) < \ell)$ of $L(n)$,
which works well at small values of $\ell$, is obtained in Section
\ref{better}, using complex analysis.
We numerically compare the two approximations in Section \ref{Num}
and finally pose some further questions in the last section.
Although, in this paper, our method has been applied to finding 
very detailed information
about the distribution function of the specific functional $L(n)$, 
many other functionals, mentioned above and in Section \ref{portmid},
can be treated analogously
if detailed information about their distribution function is desired.

\section{A portmanteau identity}
\label{portmid}
Let $N^*$ be a geometric random variable,
\[
\P(N^* = n) = w^{n-1} (1-w), \quad n \in \N,
\]
independent of the sequence $\xi_1, \xi_2, \ldots$.
We let 
\[
R^* := R(N^*-1).
\]
Thus $R^*$ is a random element of $\Z_+^*$ which is distributed like 
$R(n)$ with probability $w^n(1-w)$, for $n=0,1,\ldots$
Note that $R(0) = (0,0, \ldots)$, which is consistent with our definitions. 
To save some space, we use the abbreviations
\begin{equation}
\label{abc}
\alpha := wp, \quad 
\beta := wq, \quad
\gamma := 1-w,
\end{equation}
throughout the paper, noting that if $\alpha, \beta,\gamma$ are three
nonnegative real numbers adding up to $1$ with $\gamma$ strictly positive,
then $w, p, q=1-p$ are uniquely determined.
%and keep in mind that $\Z_+^*$ is a cone in the vector space 
%$\Z^*$ which has ``unit'' vectors $e_j$, $j=1,2,\ldots$
%Thus $e_j$ has components
%\[
%e_j(n) := \1\{n=j\}, \quad n \in \N.
%\]
%By convention,
%\[
%e_0 := (0,0,\ldots).
%\]
\begin{theorem}
\label{RECw}
For any $h : \Z_+^*\to \R$ such that $\E h(R^*)$ is defined
we have the Stein-Chen type of identity
\begin{equation}
\label{SCid}
\E h(R^*) = \gamma \sum_{j \ge 0} \alpha^j h(e_j) + \beta \sum_{j \ge 0}
\alpha^j \E h(R^* + e_j).
\end{equation}
\end{theorem}
\begin{proof}
The equation becomes apparent if we think probabilistically, using
an ``explosive coin''. Consider a usual coin (think of a British 
pound\footnote{A British pound is sufficiently thick so that the chance 
of landing on its edge is non-negligible, especially at the hands of 
a skilled coin tosser. 
If a US (thinner) nickel is used then the chance of landing 
on its edge is estimated to be $1/6000$ \cite{MS93}.})
but equip it with an explosive mechanism which is activated if the coin
touches the ground on its edge. An explosion occurs
with probability $\gamma=1-w$. When explosion occurs the coin is
destroyed immediately.
% and we do not observe heads or tails. 
As long as explosion does not occur 
then the coin lands heads or tails, as usual. Clearly,
$\alpha=wp$ is the probability that we observe heads and $\beta=wq$
is the probability that we observe tails.
We let $\explo, \heads, \tails$ denote ``explosion'', ``heads'', ``tails'',
respectively, for the explosive coin.
The possible outcomes in tossing such a coin comprise the set
\[
\Omega^*:=\bigcup_{k \ge 0} \{\heads, \tails\}^k \times \{\explo\}.
\]
Indeed, the repeated tossing of an explosive coin results in an explosion
(which may happen immediately), in which case the coin is destroyed.
%If $T_\heads$ (respectively, $T_\tails$, $T_\explo$)
%denote the number of tosses required to see the first head
%(respectively, tail, explosion)
$R^*$ can then naturally be defined on $\Omega^*$.
Let $\heads^j \explo \subset \Omega^*$ be an abbreviation for
the event of seeing heads $j$ times followed by explosion.
Similarly, for $\heads^j \tails$.
%Clearly, $\bigcup_{j \ge 0} \{\heads^j \explo\}$ and $\bigcup_{j \ge 0}
%\{\heads^j \tails\}$ and disjoint and their union is $\Omega^*$.
Clearly,
$\Omega^* = \bigcup_{j \ge 0} (\heads^j \explo) \cup \bigcup_{j \ge 0}
(\heads^j \tails)$
and all events involved in the union are mutually disjoint.
Hence
\[
\E h(R^*) = \sum_{j \ge 0} \E[ \heads^j \explo; h(R^*)]
+ \sum_{j \ge 0} \E[ \heads^j \tails; h(R^*)],
\]
where, as usual, $\E[A; Y] := \E[\1_A Y]$, if $A$ is an event and $Y$ a
random variable.
For $j\ge 0$, on the event $\heads^j\explo$, we have $R^* = e_j$.
Hence $E[ \heads^j \explo; h(R^*)] = \alpha^j \gamma h(e_j)$.
On the event $\heads^j \tails$ we have $R^* = e_j +
\theta^{j+1} R^*$, where
$\theta^{j+1} R^* = (R^*_{j+1}, R^*_{j+2}, \ldots)$,
which is independent and identical in law to $R^*$. 
Hence
$\E[ \heads^j \tails; h(R^*)] = \alpha^j \beta \E h(e_j + R^*)$.
\end{proof}

The easiest way to see that the identity we just proved
actually characterizes the law of $R^*$ is by direct computation.
If $x\in \Z_+^*$, we let
\[
z^x := z_1^{x_1} z_2^{x_2} \cdots,
\]
for any sequence $z_1, z_2, \ldots$ of real or complex numbers
such that $z_k\neq 0$ for all $k$.
(This product is a finite product, by definition of $\Z_+^*$.)
\begin{theorem}
\label{R*joint}
There is a unique (in law) random element $R^*$ of $\Z_+^*$
such that \eqref{SCid} holds for all nonnegative $h$.
For this $R^*$, we have
\[
\E z^{R^*} = 
\gamma \frac{1+\sum_{j\ge 1} \alpha^j z_j}{1-\beta-\beta \sum_{j\ge 1}
\alpha^j z_j}.
\]
Moreover, for any $\ell \in \N$, the law of $(R_1^*, \ldots, R_\ell^*)$
is specified by
\[
\E z_1^{R_1^*} \cdots z_\ell^{R_\ell^*}
= \gamma\,
\frac{1+\sum_{j =1}^\ell \alpha^j z_j + \sum_{j > \ell} \alpha^j}
{1-\beta-\beta \sum_{j=1}^\ell \alpha^j z_j -\beta \sum_{j > \ell} \alpha^j}.
\]
\end{theorem}
\begin{proof}
Let $h(x) := z^x$ in \eqref{SCid}. Then $h(e_j)= z_j$,
and $h(R^* + e_j) = z_j h(R^*)$.
Substituting into \eqref{SCid} gives the result.
Taking $z_j=1$ for all $j \ge \ell$ gives the second formula.
\end{proof}
We can now derive distributions of various functionals of $R^*$ quite easily.
For example, to deal with the one-dimensional marginals of $R^*$,
set $z_\ell=\theta$ and let $z_k=1$ for $k\neq \ell$:
\begin{equation}
\label{pgfRell*}
\E \theta^{R_\ell^*} = 
\gamma\,
\frac{1+\sum_{j \neq \ell} \alpha^j + \alpha^\ell \theta}
{1-\beta-\beta\sum_{j \neq \ell} \alpha^j -\beta \alpha^\ell \theta}.
%\equiv \gamma\,
%\frac{1+\sigma_\ell + \alpha^\ell \theta}
%{1-\beta-\beta\sigma_\ell -\beta \alpha^\ell \theta}
\end{equation}
This is a geometric-type distribution (with mass at $0$), and we give 
it a name for our convenience. 
\begin{definition}
\label{defgeo}
For $0\le \alpha, \beta \le 1$ 
let $\geo(\alpha,\beta)$ denote the probability measure $Q$ on
$\Z_+\cup \{+\infty\}=\{0,1,\ldots, +\infty\}$ with 
\[
Q\{0\} = \alpha, \quad Q\{n\} = (1-\alpha)(1-\beta) \beta^{n-1}, ~n \ge 1.
\]
\end{definition}
For example, $N^*$ has $\geo(0,w)$ distribution and $N^*-1$ has
$\geo(1-w,w)$ distribution.
Abusing notation and letting $\geo(\alpha,\beta)$ denote a random 
variable with the same law, we easily see that
\begin{align*}
&\E \geo(1-\beta,\beta) = \frac{\beta}{1-\beta}
\\
&\E\binom{\geo(\alpha,\beta)}{r} = \frac{1-\alpha}{\beta}\,
\bigg(\frac{\beta}{1-\beta}\bigg)^r, \quad r \ge 1
\\
&\E \theta^{\geo(\alpha,\beta)} =
\frac{\alpha+(1-\alpha-\beta)\theta}{1-\beta \theta}
= \frac{1-\displaystyle\frac{1-\alpha-\beta}{1-\beta}(1-\theta)}
{1+\displaystyle\frac{\beta}{1-\beta}(1-\theta)}.
\end{align*}
Therefore, comparing with \eqref{pgfRell*}, we have
\begin{corollary}
\label{Rell*coro}
$R_\ell^*$ has $\geo(\alpha_\ell, \beta_\ell)$ distribution
with
\[
\alpha_\ell = \gamma\,\frac{1+\sigma_\ell}{1-\beta-\beta\sigma_\ell},\quad
\beta_\ell = \frac{\beta\alpha^\ell}{1-\beta-\beta\sigma_\ell},
\]
where $\sigma_\ell:= \sum_{j \ge 1, j \neq \ell} \alpha^j$.
\end{corollary}
As a reality check, observe that
%\[
%\E R_\ell^* =\frac{\beta_\ell}{1-\beta_\ell} = \frac{(1-\alpha)^2 \alpha^\ell}
%{\gamma},
%\]
$\E R_\ell^* =\beta_\ell/(1-\beta_\ell) 
= (1-\alpha)^2 \alpha^\ell/\gamma$
and so
%\[
%\sum_{\ell=1}^\infty \ell \E R_\ell^* = \frac{\alpha}{\gamma}.
%\]
$\sum_{\ell=1}^\infty \ell \E R_\ell^* = \alpha/\gamma$.
On the other hand, $\sum_{\ell=1}^\infty \ell R_\ell^* =S(N^*-1)$.
Since $S(n)$ is binomial and $N^*$ is independent geometric, we have,
by elementary computations,
%\[
%S(N^*-1) \sim \geo\bigg(\frac{\gamma}{1-\beta},\, \frac{\alpha}{1-\beta}\bigg)
%\]
$S(N^*-1) \sim \geo\big(\frac{\gamma}{1-\beta},\, \frac{\alpha}{1-\beta}\big)$
and so 
%\[
%\E S(N^*-1) = \frac{\alpha}{\gamma},
%\]
$\E S(N^*-1) = \alpha/\gamma$, agreeing with the above.

As another example, consider the following functional $\overline\lambda: \Z_+^*\to \R$:
\[
\overline\lambda(x) = \sup\{i > 0:\, x_i > 0\}.
\]
%with $\overline\lambda(0) =0$, since $\sup\varnothing =0$.
\begin{corollary}
\label{L*dist}
Let $L^*:= \overline\lambda(R^*)$ be the longest run in $N^*-1$ coin tosses.
Then
\[
\P(L^* < \ell) = \frac{\gamma(1-\alpha^\ell)}{\gamma+\beta\alpha^\ell},
\quad \ell \in \N.
\]
\end{corollary}
\begin{proof}
With $0$ denoting the zero element of $\Z_+^*$, we have $\overline\lambda(0)=0$, since
$\sup \varnothing =0$. Also
\[
\overline\lambda(x+e_j) = \overline\lambda(x)\vee j, \quad j \ge 0, \quad x \in \Z_+^*.
\]
Fix $\ell \in \N$, and use \eqref{SCid}
with $h(x) := \1\{\overline\lambda(x) < \ell\}$.
Then $\P(L^* < \ell) = \E h(R^*)$.
Since $h(x+e_j) = \1\{\overline\lambda(x) \vee j < \ell\} = h(x) \1\{j < \ell\}$,
we have $\E h(R^*+e_j) = \P(L^* < \ell)\, \1\{j < \ell\}$.
Substituting into \eqref{SCid} gives
\[
\P(L^* < \ell) = \gamma \sum_{j \ge 0} \alpha^j \1\{j < \ell\}
+ \beta \sum_{j \ge 0} \alpha^j\, \1\{j < \ell\}\, \P(L^*<\ell),
\]
which immediately yields the announced formula.
\end{proof}
See also Grimmett and Stirzaker \cite[Section 5.12, Problems 46,47]{GS}
for another way
of obtaining the distribution of $L^*$.

Alternatively, we can look at the functional 
\[
\underline\lambda(x) := \inf\{i > 0:\, x_i >0\},
\]
which takes value $+\infty$ at the origin of $\Z_+^*$, 
but this poses no difficulty.
\begin{corollary}
Let $\underline\lambda(R^*)$ be the run of least length in $N^*-1$ coin tosses.
Then
\[
\P^*(\underline\lambda(R^*) \ge \ell) =
\frac{\gamma(1-\alpha+\alpha^\ell)}{\gamma-\beta(1-\alpha+\alpha^\ell)},
\quad \ell \in \N.
\]
The random variable $\underline\lambda(R^*)$ is defective
with $\P(\underline\lambda(R^*)=\infty)= \gamma(1-\alpha)/(\gamma-\beta(1-\alpha))$.
\end{corollary}
\begin{proof}
Fix $\ell \in \N$ and let $h(x) = \1\{\underline\lambda(x) \ge \ell\}$ in
\eqref{SCid}. We work out that $h(0)=1$ and, for $j \in \N$,
$h(e_j)=j$, $h(x+e_j) = h(x)\1\{j \ge \ell\}$.
The rest is elementary algebra.
\end{proof}

%Since $\Z^*$ is a vector space, there are linear functions from $\Z^*$ 
%into $\R$.
\begin{corollary}
If $h : \Z^* \to \R$ is a linear function then 
\[
\E h(R^*) = \frac{(1-\alpha)^2}{\gamma}\, \sum_{j \ge 0} \alpha^j h(e_j)
\]
\end{corollary}
%Setting $h(x) = x_\ell$ gives again the earlier formula for
%$\E R^*_\ell$.

As another example of the versatility of the portmanteau formula,
we specify the joint distribution of finitely many components
of $R^*$ together with $L^*$.
\begin{corollary}
\[
\E\big[z_1^{R_1^*}\cdots z_{\ell-1}^{R_{\ell-1}^*} ;\,L^*<\ell \big]
= \gamma\, \frac{1+\sum_{j=1}^{\ell-1} \alpha^j z_j}
{1-\beta-\beta\sum_{j=1}^{\ell-1} \alpha^j z_j}.
\]
\end{corollary}
\begin{proof}
Let
$h(x) = z_1^{x_1} \cdots z_{\ell-1}^{x_{\ell-1}} \1\{\overline\lambda(x) < \ell\}$ in
\eqref{SCid}. Then
$h(0)=0$, $h(e_j) = z_j \1\{j < \ell\}$, $h(x+e_j)= h(x) h(e_j)$,
$j \in \N$. Again, substitution into \eqref{SCid} and simple algebra
gives the formula.
\end{proof}
For verification, note that taking $\ell \to \infty$ in the last
display gives the previous formula for $\E z^{R^*}$,
while letting $z_1=\cdots=z_{\ell-1}=1$ gives the previous
formula for $\P(L^*<\ell)$.

The joint moments and binomial moments of the components of $R^*$ can be 
computed explicitly.
\begin{corollary}
\label{jointmom}
Consider positive integers $\nu$, $\ell_1 , \ldots , \ell_\nu$,
%\comments{In the previous version we assumed that\\  
%$\ell_1 < \cdots < \ell_\nu$,\\
%but I don't think we need this assumption.}
and nonnegative integers $r_1, \ldots, r_\nu$,
such that $r_0 := r_1+\cdots+r_\nu \ge 1$. Let $\bm \ell :=(\ell_1,
\ldots, \ell_\nu)$ and $\bm r := (r_1, \ldots, r_\nu)$ and set
$\bm \ell \cdot \bm r = \ell_1 r_1 + \cdots + \ell_\nu r_\nu$.
Then
\begin{equation}
\label{jointmom1}
\E z_1^{R_{\ell_1}^*} \cdots z_\nu^{R_{\ell_\nu}^*} 
=\frac{1+(1-\alpha) \sum_{j=1}^\nu \alpha^{\ell_j} (z_j-1)}
{1- \frac{(1-\alpha)\beta}{\gamma} \sum_{j=1}^\nu \alpha^{\ell_j} (z_j-1)}
\end{equation}
and
\begin{equation}
\label{jointmom2}
\E \binom{R_{\ell_1}^*}{r_1} \cdots \binom{R_{\ell_\nu}^*}{r_\nu}
=
\frac{r_0!}{r_1! \cdots r_\nu!}\,
\frac{\alpha^{\bm \ell \cdot \bm r} \beta^{r_0-1} (1-\alpha)^{r_0+1}}
{\gamma^{r_0}}
\end{equation}
%\comments{Correctness of this formula must be checked!
%Does it give the right thing when $\nu=1$?
%Lars, I didn't have the stamina ro do more computations, but I think
%you already have them, in which case, let's write a few lines.}
\end{corollary}
\begin{proof}
By Theorem \ref{RECw},
\begin{align*}
\E  
z_1^{R_{\ell_1}^*} \cdots z_\nu^{R_{\ell_\nu}^*}
&= \gamma\, \bigg( \sum_{j=1}^\nu \alpha^{\ell_j} z_j 
+ \sum_{j \not \in \{\ell_1,\ldots,\ell_\nu\}} \alpha^j \bigg)
+\beta\, \bigg( \sum_{j=1}^\nu \alpha^{\ell_j} z_j 
+ \sum_{j \not \in \{\ell_1,\ldots,\ell_\nu\}} \alpha^j \bigg)
\E  z_1^{R_1^*} \cdots z_\nu^{R_\nu^*}
\\
&= \gamma\,\bigg(\sum_{j=1}^\nu \alpha^{\ell_j} (z_j-1) 
+ \sum_{j \ge 0} \alpha^j \bigg)
+ \beta\, \bigg(\sum_{j=1}^\nu \alpha^{\ell_j} (z_j-1) 
+ \sum_{j \ge 0} \alpha^j \bigg) \E  z_1^{R_1^*} \cdots z_\nu^{R_\nu^*},
\end{align*}
from which the formula \eqref{jointmom1} follows.
Expanding the denominator in \eqref{jointmom1}, we obtain
\begin{multline*}
\E z_1^{R_{\ell_1}^*} \cdots z_\nu^{R_{\ell_\nu}^*}
= \bigg(1+(1-\alpha) \sum_{j=1}^\nu \alpha^{\ell_j} (z_j-1) \bigg)\,
\sum_{k=0}^\infty \bigg( \frac{(1-\alpha)\beta}{\gamma} \bigg)^k\, 
\bigg( \sum_{j=1}^\nu \alpha^{\ell_j} (z_j-1) \bigg)^k
\\
= 1+ \sum_{k=1}^\infty \bigg( \frac{(1-\alpha)\beta}{\gamma} \bigg)^k\,
\bigg( \sum_{j=1}^\nu \alpha^{\ell_j} (z_j-1) \bigg)^k
+ \frac{\gamma}{\beta} \,
\sum_{k=1}^\infty \bigg( \frac{(1-\alpha)\beta}{\gamma} \bigg)^k
\bigg( \sum_{j=1}^\nu \alpha^{\ell_j} (z_j-1) \bigg)^k
\\
= 1+ \frac{1-\alpha}{\beta}\,
\sum_{k=1}^\infty \bigg( \frac{(1-\alpha)\beta}{\gamma} \bigg)^k
\sum_{\substack{i_1, \ldots, i_\nu\\i_1+\cdots+i_\nu=k}}
\frac{k!}{i_1! \cdots i_\nu!}\, \alpha^{\ell_1\nu_1+\cdots+\ell_\nu i_\nu}
(z_1-1)^{i_1} \cdots (z_\nu-1)^{i_\nu}.
\end{multline*}
Now,
\begin{align*}
\E z_1^{R_{\ell_1}^*} \cdots z_\nu^{R_{\ell_\nu}^*}
&= \E (1+(z_1-1))^{R_{\ell_1}^*} \cdots (1+(z_\nu-1))^{R_{\ell_\nu}^*} 
\\
&= \sum_{i_1,\ldots, i_\nu} 
\E \binom{R_{\ell_1}^*}{r_1} \cdots \binom{R_{\ell_\nu}^*}{r_\nu}
\, (z_1-1)^{i_1} \cdots (z_\nu-1)^{i_\nu},
\end{align*}
and so formula \eqref{jointmom2} is obtained by inspection.
\end{proof}

Sometimes \cite{BK02,MP11} people are interested in the distribution 
of the number of runs exceeding a given length:
\[
G_\ell(n) := \sum_{k \ge \ell} R_k(n).
\]
Consider the $\Z_+^*$--valued random variable
\[
G(n) := (G_1(n),\, G_2(n), \ldots).
\]
We work up to a geometric random variable.
Thus, let 
\[
G^* := G(N^*-1).
\]
We can compute $\E z^{G^*}$ easily from the first formula
of Theorem \ref{R*joint} by replacing $z_j$ by $z_1\cdots z_j$:
\begin{corollary}
\[
\E z^{G^*} = \gamma\,
\frac{1+\sum_{j \ge 1} \alpha^j z_1\cdots z_j}
{1-\beta-\beta \sum_{j \ge 1} \alpha^j z_1\cdots z_j}.
\]
\end{corollary}
Marginalizing, we see that
\begin{corollary}
\label{Gell*dist}
$G_\ell^*$ has $\geo(\widetilde \alpha_\ell, \widetilde \beta_\ell)$
distribution with
%\[
%\widetilde\alpha_\ell = \frac{\gamma(1-\alpha^\ell)}{\gamma+\beta\alpha^\ell},
%\quad
%\widetilde\beta_\ell= \frac{\beta\alpha^\ell}{\gamma+\beta\alpha^\ell}
%\]
$\widetilde\alpha_\ell 
= \gamma(1-\alpha^\ell)/(\gamma+\beta\alpha^\ell)$,
$\widetilde\beta_\ell
= \beta\alpha^\ell/(\gamma+\beta\alpha^\ell)$.
\end{corollary}
%This follows from direct comparison of the $\E\theta^{G_\ell^*}$
%with the formula for the probability generating function of
%a $\geo(\alpha, \beta)$ random variable.
%As another example of getting information about the dependence
%between the components of $G^*$, we compute explicitly
%the joint binomial moments. 
%\begin{corollary}
%\label{jointmom}
%Consider positive integers $\nu$, $0 < \ell_1 < \cdots < \ell_\nu$, 
%and nonnegative integers $r_1, \ldots, r_\nu$,
%such that $r_0 := r_1+\cdots+r_\nu \ge 1$. Let $\bm \ell :=(\ell_1,
%\ldots, \ell_\nu)$ and $\bm r := (r_1, \ldots, r_\nu)$ and set
%$\bm \ell \cdot \bm r = \ell_1 r_1 + \cdots + \ell_\nu r_\nu$.
%Then
%\[
%\E \binom{G_{\ell_1}^*}{r_1} \cdots \binom{G_{\ell_\nu}^*}{r_\nu}
%=
%\frac{r_0!}{r_1! \cdots r_\nu!}\,
%\frac{\alpha^{\bm \ell \cdot \bm r} \beta^{r_0-1} (1-\alpha)^{r_0+1}}
%{\gamma^{r_0}}
%\]
%\comments{Correctness of this formula must be checked!
%Does it give the right thing when $\nu=1$? 
%Lars, I didn't have the stamina ro do more computations, but I think
%you already have them, in which case, let's write a few lines.}
%\end{corollary}

\section{Number of runs of given (or exceeding a given) 
\label{exceed}
length in $n$ coin tosses}
Our interest next is in obtaining information about the
distributions of $R_\ell(n)$ and $G_\ell(n)$.
Since $R_\ell^*$ and $G_\ell^*$ are both 
of $\geo(\alpha, \beta)$ type with explicitly known 
parameters, and since\footnote{If $X$ is a random variable,
we let $\mathcal L \{X\}$ be its law.}
\[
\mathcal L\{G_\ell^*\} = (1-w) \sum_{n\ge 0} w^n\, \mathcal L\{G(n)\}
\]
(likewise for $R_\ell^*$),
the problem is, in principle, solved. Moreover, such formulas exist in
the numerous references. See, e.g., \cite{MU96,BK02}. Our intent in this
section is to give an independent  derivation of the formulas but
also point out their relations with hypergeometric functions.

It turns out that (i) formulas for $G_\ell(n)$ are simpler than
those for $R_\ell(n)$ and (ii) binomial moments for both variables
are simpler to derive than moments.
%Let us start by computing expectations.
%From Corollary 7 ???, and the the definition of a $\geo(\alpha,\beta)$
%random variable, we have
%\[
%\E G_\ell^* = \frac{1-\widetilde\alpha_\ell}{1-\widetilde\beta_\ell}
%= \frac{(1-\alpha)\alpha^\ell}{\gamma}
%= \frac{(1-wp) (wp)^\ell}{1-w}.
%\]
We therefore start by computing the $r$-th binomial moment of $G_\ell(n)$.
By Corollary \ref{Gell*dist}, $G_\ell^*$ is a $\geo(\widetilde\alpha_\ell,
\widetilde\beta_\ell)$ random variable, and, from the formulas
following Definition \ref{defgeo}, we have
\begin{align*}
\E \binom{G_\ell^*}{r} 
&= \frac{1-\widetilde\alpha_\ell}{\widetilde\beta_\ell} 
\, \bigg(\frac{\widetilde\beta_\ell}{1-\widetilde\beta_\ell} \bigg)^r
= \frac{1-\alpha}{\beta}
\bigg(\frac{\beta\alpha^\ell}{\gamma}\bigg)^r
\\
&= (1-wp) (wq)^{r-1} (wp)^{\ell r} (1-w)^{-r}
\\
&= (1-w) p^{\ell r} q^{r-1} \times 
\underbrace{(1-wp) w^{\ell r+r-1} (1-w)^{-r-1}} .
\end{align*}
Now use the Taylor expansion
\begin{equation}
\label{TaylorBinomial}
(1-w)^{-r-1} = \sum_{k=0}^\infty \binom{r+k}{r} w^k
\end{equation}
to express the under-braced term above as 
\begin{align*}
(1-wp) \sum_{k=0}^\infty \binom{r+k}{r} w^{\ell r + r -1 +k}
&= \sum_{k=0}^\infty \binom{r+k}{r} w^{\ell r + r -1 +k}
- p \sum_{k=0}^\infty \binom{r+k}{r} w^{\ell r + r  +k}
\\
&= \sum_n \binom{n+1-\ell r}{r} w^n
- p \sum_n \binom{n-\ell r}{r} w^n
\\
&=
\sum_n \bigg[ \binom{n+1-\ell r}{r} - p \binom{n-\ell r}{r}\bigg]\, w^n.
\end{align*}
So, by inspection,
\begin{equation}
\label{Gellnbinom}
\E\binom{G_\ell(n)}{r} = p^{\ell r} q^{r-1}
\bigg[ \binom{n+1-\ell r}{r} - p \binom{n-\ell r}{r}\bigg].
\end{equation}
In particular, we have
\[
\E G_\ell(n) = p^\ell [(n-\ell+1)-p(n-\ell)], \quad n \ge \ell,
\]
and, since $R_\ell(n) = G_\ell(n) -  G_\ell(n+1)$,
\[
\E R_\ell (n) =
p^\ell [(n-\ell+1) - 2(n-\ell) p + (n-\ell-1)p^2], \quad n > \ell,
\]
while $\E R_n(n)=p^n$. Notice that
\[
\lim_{n \to \infty} \frac{1}{n}\, \E R_\ell(n) = p^\ell q^2,
\]
as expected by the ergodic theorem.

We now use the standard formula relating probabilities to binomial moments:
\footnote{The binomial coefficient $\binom{a}{b}$ is taken to be zero if $b
> a$ or if $a < 0$.}
\begin{equation}
\label{probmom}
\P(G_\ell(n) = x) = \sum_{r \ge x} (-1)^{r-x} \binom{r}{x} \E\binom{G_\ell(n)}{r}.
\end{equation}
Substituting the formula for the binomial moment and changing variable
from $r \ge x$ to $m=r-x \ge 0$
we obtain
\begin{multline*}
\P(G_\ell(n) = x) 
%= \sum_{r\ge x} (-1)^{r-x} \binom{r}{x} p^{\ell r} q^{r-1}
%\bigg[ \binom{n+1-\ell r}{r} - p \binom{n-\ell r}{r}\bigg]
%\\
 = \sum_{m \ge 0} (-1)^m \binom{x+m}{x} p^{\ell(x+m)} q^{x+m-1} 
\bigg[ \binom{n+1-\ell (x+m)}{x+m} - p \binom{n-\ell (x+m)}{x+m}\bigg]
\\
 = p^{\ell x} q^{x-1} 
\bigg[\sum_{m \ge 0} (-p^\ell q)^m \binom{x+m}{x} \binom{n+1-\ell (x+m)}{x+m}
- p \sum_{m \ge 0} (-p^\ell q)^m \binom{x+m}{x} \binom{n-\ell (x+m)}{x+m}\bigg].
%\\
 %= p^{\ell x} q^{x-1} 
%\bigg[\sum_{m \ge 0} (-p^\ell q)^m \binom{x+m}{x} \binom{x+y-\ell m}{x+m}
%- p \sum_{m \ge 0} (-p^\ell q)^m \binom{x+m}{x} \binom{x+y-\ell m-1}{x+m}\bigg]
%\\
%y=(n+1)-(\ell+1)x
\end{multline*}
It is interesting to notice the relation of the distribution
of $G_\ell(n)$ to hypergeometric functions. Recall the notion of
the hypergeometric function \cite[Section 5.5.]{GKP} (the notation is 
from this book and is not standard):
\[
F\bigg(\begin{matrix}a_1, \ldots, a_m\\ b_1,\ldots,b_n\end{matrix}
\bigg|\, z \bigg) =
\sum_{k \ge 0} \frac{a_1^{\overline k}\cdots a_m^{\overline k}}
{b_1^{\overline k}\cdots b_n^{\overline k}}\, \frac{z^k}{k!},
\]
where $m, n \in \Z_+$, $a_1, \ldots, a_m \in \C$, $b_1, \ldots, b_n
\in \C \setminus\{0,-1,-2,\ldots\}$, $z \in \C$, and $x^{\overline k}
:= x(x+1) \cdots (x+k-1)$. 
A little algebra gives
\begin{multline}
H_\ell(x,y; z):= 
\sum_{m \ge 0} z^m\, \binom{x+m}{x} \binom{x+y-\ell m}{x+m}
= \binom{x+y}{x}\,
F\bigg(\begin{matrix}
\bm V_{\ell+1}(y)
\\ 
\bm V_\ell(x+y)
\end{matrix}
\,
\bigg|\, -\frac{(\ell+1)^{\ell+1}}{\ell^\ell} \,z \bigg),
\label{Hell}
\end{multline}
where $\bm V_{\ell+1}(y)$ and $\bm V_\ell(x+y)$ denote arrays of sizes
$\ell+1$ and $\ell$ respectively, defined via
\[
\bm V_k(u) := -\frac{1}{k} \big(u,\, u-1, \ldots, u-k+1\big).
\]
Looking back at \eqref{probmom}  we recognize
that the two terms in the bracket are expressible in terms
of the function $H_\ell$:
\[
\P(G_\ell(n)=x) = p^{\ell x} q^{x-1}
\big[
H_\ell(x,\, n+1-(\ell+1)x\, -p^\ell q)
-H_\ell(x,\, n-(\ell+1)x\, -p^\ell q)
\big].
\]
The point is that the probabilities $\P(G_\ell(n)=x)$
are expressible in terms of the function $H_\ell$ which 
is itself expressible in terms of a hypergeometric function
as in \eqref{Hell}.
Hypergeometric functions are efficiently computable 
via computer algebra systems 
(we use Maple\texttrademark.)

Ultimately, the hypergeometric functions appearing above are nothing
but polynomials. So the problem is, by nature, of combinatorial character.
Instead of digging in the literature for recursions for these functions,
we prefer to transform the portmanteau identity into a recursion
which can be specialized and iterated.

\section{Portmanteau recursions in the time domain}
\label{portmrec}
Recall the identity \eqref{SCid}. 
We pass from ``frequency domain'' (variable ``$w$'') to
``time domain'' (variable ``$n$''), we do obtain a veritable
recursion in the space $\Z_+^*$. Recalling that $\alpha, \beta, \gamma$
are given by \eqref{abc}
and that
%\[
%\mathcal L \{h(R^*)\} = \sum_{n \ge 0}
%(1-w) w^n \, \mathcal L \{h(R(n)\},
%\]
$\mathcal L \{h(R^*)\} = \sum_{n \ge 0}
(1-w) w^n \, \mathcal L \{h(R(n)\}$,
we take each of the terms in \eqref{SCid} and bring out its dependence on
$w$ explicitly. The left-hand side of \eqref{SCid} is
\begin{equation}
\label{LHS}
\E h(R^*) %= \gamma \sum_{j\ge 0} \alpha^j h(e_j)
= (1-w) \sum_{n\ge 0} w^n\,\E h(R(n)).
\end{equation}
The first term on the right-hand side of \eqref{SCid} is
\begin{equation}
\label{RHS1}
\gamma \sum_{n \ge 0} \alpha^j h(e_n) 
= (1-w) \sum_{n \ge 0} w^n p^n h(e_n).
\end{equation}
As for the second term of  \eqref{SCid}, we have
\begin{align*}
\beta \sum_{j \ge 0} \alpha^j \E h(R^*+e_j)
&= wq \sum_{j \ge 0} w^j p^j (1-w)\sum_{n \ge 0} w^n \E h(R(n)+e_j)
\\
&= (1-w)q \sum_{j \ge 0} \sum_{n \ge 0} w^{1+j+n}  p^j \E h(R(n)+e_j)
\end{align*}
Change variables by
%\[
%(j,\, n) \mapsto (j,\, m=1+j+n)
%\]
$(j,\, n) \mapsto (j,\, m=1+j+n)$
to further write 
\begin{align}
\beta \sum_{j \ge 0} \alpha^j \E h(R^*+e_j)
&= (1-w) q \sum_{m \ge 0} \sum_{0\le j \le m-1} w^m p^j \E h(R(m-j-1)+e_j)
\nonumber
\\
&= (1-w)  \sum_{n \ge 0} w^n 
q \sum_{0\le j \le n-1} p^j \E h(R(n-j-1)+e_j).
\label{RHS2}
\end{align}
Using \eqref{SCid} and \eqref{LHS}, \eqref{RHS1}, \eqref{RHS2}, we obtain
\begin{theorem}
\label{RECn}
Let $h: \Z_+^* \to \R$ be any function.
Then, for all $n \in \N$,
\begin{equation}
\label{portrecn}
\E h(R(n)) = q \sum_{j=0}^{n-1} p^j \E h(R(n-j-1)+e_j) + p^n h(e_n).
\end{equation}
\end{theorem}
\begin{remark}
(i) We say ``any function'' because $R(n)$ takes finitely many values for
all $n$.
\\
(ii) This is a linear recursion but, as expected, it does not have bounded 
memory. 
\\
(iii) It can easily be programmed. It is initialized with
$\E h(R(0)) = h(0)$.
\\
(iv) Of course, this recursion is nothing else but ``explicit counting''.
\\
(v) One could provide an independent proof of Theorem \ref{RECn}
and obtain the result of Theorem \ref{RECw}. This is a matter of
taste.
\\
(vi) We  asked Maple to run the recursion a few times and here
is what it found:
\begin{align*}
& \E h(R(1)) = qh(0)+ p h(e_1)
\\
& \E h(R(2))= q^2 h(0)+ 2qp h(e_1) + p^2 h(e_2)
\\
& \E h(R(3)) = q^3 h(0) + 3q^2 p h(e_1) + qp^2 h(2e_1)
+ 2 qp^2 h(e_2) + p^3 h(e_3)
\\
&\E h(R(4)) = q^4 h(0) + 4 q^3 p h(e_1) + 3q^2p^2 h(2e_1) + 3 q^2 p^2 h(e_2)
+2 qp^3 h(e_1+e_2) + 2 q p^3 h(e_3) \\
&\quad \qquad  \qquad + p^4 h(e_4),
\end{align*}
%\comments{Should we try to interpret these terms combinatorially?}
which could be interpreted combinatorially.\footnote{It may be worth
carrying out the combinatorial approach further.}
\end{remark}

Since $G(n) = \sigma(R(n))$ where $\sigma : \Z_+^* \to \Z_+^*$ is given by
\[
\sigma(x)_k := \sum_{j \ge k} x_j,
\]
if $f: \Z_+^* \to \R$ is any function then, letting $h= f \comp \sigma$ in
the recursion of Theorem \ref{RECn}, and noting that
$\sigma(e_n) = e_1+\cdots+e_n$, we have
\begin{corollary}
Let $f: \Z_+^* \to \R$ is any function. Then, for all $n \in \N$,
\[
\E f(G(n)) = q \sum_{j=0}^{n-1} p^j \E f(G(n-j-1)+e_j) + p^n f(e_1+
\cdots+e_n).
\]
\end{corollary}
These two recursions can be transformed into recursions for probability
generating functions. Recalling that $z^x = \prod_{j \ge 1} z_j^{x_j}$,
for $x \in  \Z^*$, we consider
\[
\Phi_n(z) := \E z^{R(n)},
\quad
\Psi_n(z) := \E z^{G(n)},
\quad
z \in \C^{\N},
\]
and immediately obtain
\begin{corollary}
The probability generating functions $\Phi_n$ and $\Psi_n$
of the random elements $R(n)$ and $G(n)$, respectively, of $\Z_+^*$
satisfy $\Phi_0(z) = \Psi_0(z)=1$, and, for $n \in \N$,
\begin{align*}
\Phi_n(z) &= q \Phi_{n-1}(z) + 
q \sum_{1 \le j \le n-1} p^j z_j \Phi_{n-j-1}(z)
+ p^n z_n
\\
\Psi_n(z) &= q \Psi_{n-1}(z) + 
q \sum_{1 \le j \le n-1} p^j z_1\cdots z_j \Psi_{n-j-1}(z)
+ p^n z_1\cdots z_n.
\end{align*}
\end{corollary}
Let us now look at $G_\ell(n)$. 
Consider the probability generating function
\[
\Psi_{n,\ell}(\theta) := \E \theta^{G_\ell(n)},
\quad \theta \in \C.
\]
Clearly, $\Psi_{n,\ell}(\theta)=1$, for $n < \ell$
and
%\[
%\Psi_{n,\ell}(\theta) = \Psi_n(\bm 1 + (\theta-1) e_\ell)
%\]
$\Psi_{n,\ell}(\theta) = \Psi_n(\bm 1 + (\theta-1) e_\ell)$,
where $\bm 1 \in \Z^\N$ is the infinite repetition of $1$'s.
We thus have
\begin{corollary}
For  $n < \ell$, we have $\Psi_{n,\ell}(\theta)=1$, and
\begin{multline*}
\Psi_{n,\ell} =
\begin{cases}
q\displaystyle \sum_{j=0}^{n-\ell-1} p^j \Psi_{n-j-1, \ell}
+ p^{n-\ell} + (\theta-1) p^\ell, &  \ell \le n \le 2\ell
\\
q \displaystyle\sum_{j=0}^{\ell-1} p^j \Psi_{n-j-1, \ell}
+ q\theta \sum_{j=\ell}^{n-\ell-1} p^j \Psi_{n-j-1, \ell}
+ q\theta (p^{n-\ell}-p^\ell) + \theta p^m, &  n \ge 2\ell+1
\end{cases}
\end{multline*}
%\comments{Maybe a little more info on the derivation is needed.
%I made no attempt to ``solve'' the recursions. Similar
%recursion is possible for $\Phi_{n,\ell}$. Shall we add it?}
\end{corollary}

\section{Longest run, Poisson and other approximations}
\label{lrpa}
Recall that $L(n) = \overline\lambda(R(n))$ is the length of the longest run
in $n$ coin tosses.
Although there is an explicit formula (see Corollary \ref{L*dist}) for 
\begin{equation}
\label{Lstar}
(1-w) \sum_{n=0}^\infty w^n \P(L(n) < \ell) =
\P(L^* < \ell) = \frac{(1-w) (1-(wp)^\ell)}{1-w+(wq)(wp)^\ell},
\end{equation}
inverting this does not result into explicit expressions. To see what
we get, let us, instead, note that
\[
\P(L(n) < \ell) = \P(G_\ell(n)=0) = \sum_{r \ge 0}
(-1)^r \, \E \binom{G_\ell(n)}{r}
\]
and use the binomial moment formula \eqref{Gellnbinom} to obtain
\begin{equation}
\label{Fellexact}
F_\ell(n):=\P(L(n) < \ell) = 1+ \sum_{r \ge 1} (-1)^r 
\color{blue} 
\bigg[ \binom{n-\ell r}{r}p^{\ell r} q^{r} 
+ \binom{n-\ell r}{r-1} p^{\ell r} q^{r-1} \bigg].
\end{equation}
It is easy to see the function $n \mapsto F_\ell(n)$ satisfies a recursion.
\begin{proposition}
Let $\ell \in \N$. Define $F_\ell(0) =1$ and, for
$n \ge 1$,  $F_\ell(n)=\P(L(n) < \ell)$.
Then
\[
F_\ell(n) = q F_\ell(n-1) + qp F_\ell(n-2) +\cdots+ qp^{\ell-1} F_\ell(n-\ell).
\]
\end{proposition}
\begin{proof}
This can be proved directly by induction. But, since Theorem \ref{RECn} 
is available, set $h(x) := \1\{\overline{\lambda}(x) < \ell\}$, observe
that $h(x+e_j) = h(x) \1\{j < \ell\}$ and substitute into \eqref{portrecn}.
\end{proof}

\subsection{The Poisson regime for large lengths}
According to Feller \cite[Section XIII.12, Problem 25, page 341]{FEL68},
asymptotics for $L(n)$ go back to von Mises \cite{VM1921}.
Very sharp asymptotics for $L(n)$ are also known; 
see Erd\H{o}s and R\'enyi \cite{ER70}, its sequel paper by Erd\H{o}s, and
R\'ev\'esz \cite{ER76} and the review paper by R\'ev\'esz \cite{R78}.
But it is a matter of elementary analysis to see that the 
distribution function $\ell \mapsto F_\ell(n)$ exhibits a cutoff %phenomenon
at $\ell$ of the order of magnitude of $\log n$.
To see this in a few lines, consider the formula \eqref{Gellnbinom} for
the binomial moment of $G_\ell(n)$. Then
\begin{lemma}
\label{llll}
Keep $0 < p < 1$ fixed and let $\ell=\ell(n) \to \infty$ so that
%\[
%n p^{\ell(n)} q \to \theta,
%\]
$n p^{\ell(n)} q \to \theta$,
as $n \to \infty$, for some $\theta > 0$.
Then
\[
\E \binom{G_{\ell(n)}(n)}{r} \to \frac{\theta^r}{r!},
\]
and
\[
\P(G_{\ell(n)}(n) = 0) \to e^{-\theta}.
\]
\end{lemma}
%\begin{proof}
%Expanding the binomial coefficients in \eqref{Gellnbinom},
%\begin{align*}
%\E \binom{G_{\ell(n)}(n)}{r}  &=
%\frac{1}{r!} \prod_{j=0}^{r-1} \big[ (n-{\ell(n)} r-j) p^{\ell(n)} q \big]
%+ \frac{1}{(r-1)!} \prod_{j=0}^{r-2} \big[ (n-{\ell(n)} r-j) p^{\ell(n)} q \big]
%\times p^{\ell(n)}
%\\
%&\to \frac{1}{r!} \theta^r + \frac{1}{(r-1)!} \theta^{r-1} \times 0
%= \frac{\theta^r}{r!}, \quad \text{ as } n \to \infty.
%\end{align*}
%For the second assertion, use \eqref{probmom}:
%\[
%\P(G_{\ell(n)}(n)=0) = \sum_{r \ge 0} (-1)^r \E \binom{G_{\ell(n)}(n)}{r}
%\to \sum_{r \ge 0} (-1)^r \frac{\theta^r}{r!}
%= e^{-\theta}, \quad \text{ as } n \to \infty.
%\]
%\end{proof}
The proof is elementary.
Since $\P(G_\ell(n)=0) = \P(L(n) < \ell)$ for all $n$ and $\ell$,
the last asymptotic result can be translated immediately into the following
threshold behavior:
\begin{corollary}
\label{coroapp}
Let $0 < \alpha < \infty$, $0 < \beta \le +\infty$.
Then
\[
\P(L(n)<\alpha \log_{1/p} n + \log_{1/p} \beta) \to
\begin{cases}
e^{-q/\beta} &, \text{ if } \alpha=1 
\\
1 &, \text{ if } \alpha > 1
\\
0 &, \text{ if } \alpha < 1
\end{cases}
.
\]
\end{corollary}
In Figure \ref{fig1}, we take $p=1/2$
and plot $\ell \mapsto \P(L(n) \ge \ell)$
for three values of $n$.
%$n=10,100$ and $1000$ and $p=1/2$. Notice that the ``drop'' occurs at
%$\ell=\log_2 10 \approx 3.3$, $\log_2 100 \approx 6.6$ and $\log_2 1000
%\approx 9.7$, respectively.
\begin{figure}[h]
\begin{center}
\epsfig{file=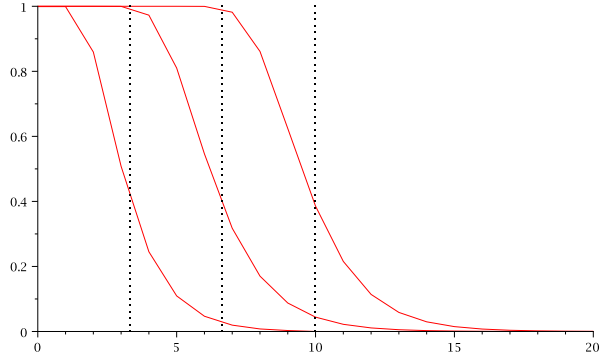,width=8cm}
\caption{\small \em Plot (piecewise linear interpolation)
of $\ell\mapsto \P(L(n) \ge \ell)$ for $n=10, 100, 1000$
and $p=1/2$. The vertical lines are at $\ell=\log_2(n)$.}
\label{fig1}
\end{center}
\end{figure}
%%%%%%%%%%%%%%%%
\begin{table}[h]
\label{table1}
\begin{center}
\begin{minipage}{12cm}
\caption{\small \em 
Comparing exact and approximate values for $\P(L(n) \ge \ell)$
when $p=1/2$ and $n=10^4$}
\end{minipage}
\\[4mm]
\begin{tabular}{l|l|l}
$\ell$ & $\P(L(n) \ge  \ell)$ & $1-\exp(-(n-\ell) p^\ell q-p^\ell)$
\\\hline
10 & 0.992583894386551 & 0.992394672192560
%\\
%11 & 0.913367688920047 & 0.912770175666911
\\
12 & 0.705167040532444 & 0.704616988848744
%\\
%13 & 0.456748458590744 & 0.456475326366319
\\
14 & 0.262835671849087 & 0.262736242068365
%\\
%15 & 0.141377186760083 & 0.141346252684305
\\
20 & 0.004748524931253 & 0.004748478671106 
\\
50 & 4.41957581641815 $\times 10^{-12}$ & 4.42000000000001 $\times 10^{-12}$
\end{tabular}
\end{center}
\end{table}
%%%%%%%%%%%%%%%%

Corollary \ref{coroapp}
suggests the practical approximation
\begin{equation}
\label{aapp}
\P(L(n) < \ell) \doteq \exp(- \E G_\ell(n)) 
= \exp(-(n-\ell) p^\ell q-p^\ell),
\end{equation}
valid for large $n$ and $\ell$, roughly when $\ell$ is of order 
$\log_{1/p} n$ or higher.
In table \ref{table1} we compare the exact result with the
approximation for $n=10^4$,
$p=1/2$, and $\ell$ ranging from slightly below $\log_2 10^4
\approx 13.288$ to much higher values.
We programmed \eqref{Fellexact} in Maple to obtain the exact values of
$\P(L(n) \ge \ell)$. 
%We note that for values of $\ell$ smaller
%than $10$, a straightforward coding of \eqref{Fellexact} is too time-consuming
%and refer to Section \ref{better} for better approximations in this case.

In Figure \ref{fig2} we plot $\ell\mapsto \P(L(n) \ge \ell)$ for $n=1000$
and three different values of $p$. We also plot the analytical 
approximation given by the right-hand side of \eqref{aapp}.
Notice that, visually at least, there is no way to tell the
difference between real values and the approximating curves.
\begin{figure}[h]
\begin{center}
\epsfig{file=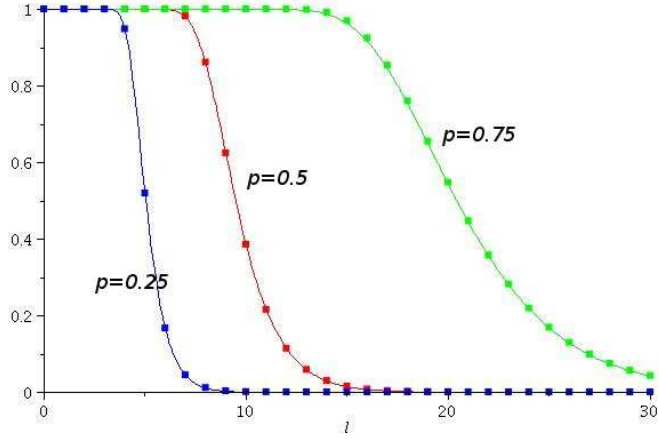,width=9cm}
\caption{\small \em Plot of $\ell\mapsto \P(L(1000) \ge \ell)$ 
and $p=0.25,0.5, 0.75$. The dots correspond to the actual
values. The solid lines correspond to the analytical
approximation \eqref{aapp}} 
\label{fig2}
\end{center}
\end{figure}

The result of Lemma \ref{llll} easily implies that
the law of $G_{\ell(n)}(n)$ converges weakly, as $n \to \infty$
to a Poisson law with mean $\theta$. 
\begin{corollary}
Under the assumptions of Lemma \ref{llll}, we have
\[
\mathcal L\{G_{\ell(n)}(n)\} \to \operatorname{Poisson}(\theta).
\]
%and
%\[
%\mathcal L\{R_{\ell(n)}(n)\} \to \operatorname{Poisson}(\theta q)
%\]
\end{corollary}
\begin{proof}
It is enough to establish convergence of binomial moments to those
of a Poisson law. Recall that if $N$ is $\operatorname{Poisson}(\theta)$
then $\E\binom{N}{r} = \theta^r/r!$. Lemma \ref{llll} tells us that
the $r$-th binomial moment of $G_{\ell(n)}(n)$ converges to $\theta^r/r!$
and this establishes the result.
\end{proof}

More interestingly, using the result of Corollary \ref{jointmom}, we
arrive at
\begin{proposition}
\label{PoiApp}
Consider $\nu \in \N$, positive real numbers
$\theta_1,\ldots,\theta_\nu$, and
sequences  $\ell_j(n)$, $j =1, \ldots, \nu$
of positive integers, such that
\[
\lim_{n \to \infty}n p^{\ell_j(n)} q = \theta_j, \quad j=1,\ldots,\nu.
\]
Then
\[
\mathcal L \{R_{\ell_1(n)}(n), \ldots, R_{\ell_\nu(n)}(n)\}
\to
\operatorname{Poisson}(\theta_1 q)  \times \cdots
\times \operatorname{Poisson}(\theta_\nu q),
\]
as $n \to \infty$.
\end{proposition}
\begin{proof}
It suffices to show that the joint binomial moments converge
to the right thing.
%, namely, that
%\[
%\lim_{n \to \infty}
%\E \binom{R_{\ell_1(n)}(n)}{r_1} \cdots \binom{R_{\ell_\nu(n)}(n)}{r_\nu}
%= \frac{(\theta_1 q)^{r_1}}{r_1!}
%\cdots \frac{(\theta_\nu q)^{r_\nu}}{r_\nu!},
%\]
%for all nonnegative integers $r_1, \ldots, r_\nu$.
Fix $\ell_1, \ldots, \ell_\nu$, $r_1,\ldots,r_\nu$,
set $r_0=r_1+\cdots+r_\nu$, and, using the abbreviations
\eqref{abc} for $\alpha, \beta$ and $\gamma$, write the expression
\eqref{jointmom2} for the joint binomial moments as
\[
\E \binom{R_{\ell_1}^*}{r_1} \cdots \binom{R_{\ell_\nu}^*}{r_\nu}
=
(1-w)\,
\frac{r_0!}{r_1! \cdots r_\nu!}\,
p^{\bm \ell \cdot \bm r} q^{r_0-1}
w^{\bm \ell \cdot \bm r + r_0-1}
\frac{(1-wp)^{r_0+1}}
{(1-w)^{r_0+1}}
\]
Expand $(1-wp)^{r_0+1}$ using the binomial formula, and
$(1-w)^{r_0+1}$ using \eqref{TaylorBinomial} to write
\[
\frac{(1-wp)^{r_0+1}}{(1-w)^{r_0+1}}
= \sum_{k=0}^\infty \sum_{s=0}^{r_0+1}
(-p)^s w^{k+s} \binom{r_0+1}{s} \binom{r_0+k}{r_0}
\]
and obtain
\begin{align*}
\E \binom{R_{\ell_1}^*}{r_1} \cdots \binom{R_{\ell_\nu}^*}{r_\nu}
&=
(1-w)\,
\frac{r_0!}{r_1! \cdots r_\nu!}\,
p^{\bm \ell \cdot \bm r} q^{r_0-1}
%w^{\bm \ell \cdot \bm r + r_0-1}
\sum_{k=0}^\infty \sum_{s=0}^{r_0+1}
(-p)^s w^{\bm \ell \cdot \bm r + r_0-1+k+s} \binom{r_0+1}{s} \binom{r_0+k}{r_0}
\\
&=
(1-w)\, \sum_{n=0}^\infty w^n
\,
\sum_{s=0}^{r_0+1}
\frac{r_0!}{r_1! \cdots r_\nu!}\,
p^{\bm \ell \cdot \bm r} q^{r_0-1}
(-p)^s \binom{r_0+1}{s} \binom{n+1-\bm \ell \cdot \bm r-s}{r_0}.
\end{align*}
Therefore,
\begin{align*}
\E \binom{R_{\ell_1}(n)}{r_1} \cdots \binom{R_{\ell_\nu}(n)}{r_\nu}
&=
\frac{r_0!}{r_1! \cdots r_\nu!}\,
p^{\bm \ell \cdot \bm r} q^{r_0-1}
\sum_{s=0}^{r_0+1}
(-p)^s \binom{r_0+1}{s} \binom{n+1-\bm \ell \cdot \bm r-s}{r_0}
\\
&= \frac{q^{-1}}{r_1! \cdots r_\nu!}\,
\sum_{s=0}^{r_0+1}
(-p)^s \binom{r_0+1}{s}
(n+1-\bm \ell \cdot \bm r-s)_{r_0}
\prod_{j=1}^\nu (p^{\ell_j} q)^{r_j},
\end{align*}
where $(N)_{r_0}= N(N-1)\cdots(N-{r_0}+1)$.
Using the assumptions, we have
\[
\lim_{n \to \infty} (n+1-\bm \ell \cdot \bm r-s)_{r_0}
\prod_{j=1}^\nu (p^{\ell_j} q)^{r_j}
= \theta_1^{r_1} \cdots \theta_\nu^{r_\nu} ,
\]
and so
\begin{align*}
\lim_{n \to \infty} 
\E \binom{R_{\ell_1(n)}(n)}{r_1} \cdots \binom{R_{\ell_\nu(n)}(n)}{r_\nu}
&= q^{-1} \sum_{s=0}^{r_0+1} (-p)^s \binom{r_0+1}{s}\,
\frac{\theta_1 ^{r_1}}{r_1!} \cdots \frac{\theta_\nu ^{r_\nu}}{r_\nu!}
\\
&= q^{-1} (1-p)^{r_0+1}\, 
\frac{\theta_1 ^{r_1}}{r_1!} \cdots \frac{\theta_\nu ^{r_\nu}}{r_\nu!}
\\
&= q^{r_1+\cdots+r_\nu} \, 
\frac{\theta_1 ^{r_1}}{r_1!} \cdots \frac{\theta_\nu ^{r_\nu}}{r_\nu!},
\end{align*}
establishing the assertion.
\end{proof}

\subsection{A better approximation for small length values}
\label{better}
We now pass on to a different approximation for $F_\ell(n)=\P(L(n)<\ell)$.
Consider again \eqref{Lstar},
\begin{equation}
\label{wtransf}
\sum_{n = 0}^\infty w^n\, F_\ell(n)
= \frac{1-(wp)^\ell}{1-w+(wq)(wp)^\ell},
\end{equation}
and look at the denominator
\[
f(w) := 1-w + p^\ell q\, w^{\ell+1},
\]
considered as a polynomial in $w \in \C$, of degree $\ell+1$.
The smallest (in magnitude) zeros of $f(w)$ govern the behavior of
$n  \mapsto F_\ell(n)$, for $n$ large (and all $\ell$.)
\begin{proposition}
\label{roote}
The equation
\[
f(w)=0, \quad w \in \C,
\]
has two real roots $w_0=w_0(\ell)$ and $1/p$, such that
\begin{align*}
&1< w_0 < \frac{\ell+1}{\ell} < \frac{1}{p((\ell+1)q)^{1/\ell}} < \frac{1}{p},
&p < \frac{\ell}{\ell+1}
\\
&1 < \frac{1}{p} < \frac{\ell+1}{\ell} < \frac{1}{p((\ell+1)q)^{1/\ell}}
< w_0, 
&p > \frac{\ell}{\ell+1}
\\
&1< w_0 = \frac{1}{p}, &p = \frac{\ell}{\ell+1},
\end{align*}
and all other roots are outside the circle with radius $\max(w_0, 1/p)$
in the complex plane.
Moreover, $\lim_{\ell \to \infty} w_0(\ell)=1$.
\end{proposition}
\proof
We check the behavior of $f(w)$ for real $w$. First, we have
$f(1/p)=0$. 
Now,
$f'(w) = -1+ (\ell+1) p^\ell q w^\ell$, and
so the only real root of $f'(w)=0$ is $w^*=1/p((\ell+1)q)^{1/\ell}$. 
Since $f''(w) = \ell(\ell+1) p^\ell q w^{\ell-1}$, the function $f$ is
strictly convex on $[0,\infty)$ and so $f(w^*)$ is a global 
minimum of $f$ on $[0,\infty)$.
Notice that $f(w^*) = 1- \ell w^*/(\ell+1)$.
We claim that $f(w^*)\le 0$, or, equivalently, that $w^* \ge  (\ell+1)/\ell$.
Upon substituting with the value of $w^*$, this last inequality is
equivalent to $p^\ell(1-p) \le (\ell/(\ell+1))^\ell$.
But this is true, since
$\max_{0 \le p \le 1}p^\ell(1-p) = \ell^\ell/(\ell+1))^{\ell+1}
\le (\ell/(\ell+1)^\ell)$.
Hence $f(w^*) \le 0$ with equality if and only if $p=\ell/(\ell+1)$.
On the other hand, $f(0)=1$ and $\lim_{w\to \infty}
f(w) = \infty$. Therefore $f(w)=0$ has two positive real roots straddling
$w^*$. One of them is $1/p$. Denote the other root by $w_0$.
Since $f((\ell+1)/\ell) = -\frac{1}{\ell}+((\ell+1)/\ell)^{\ell+1}
p^\ell q < 0$, and $f(w^*) < 0$, provided that $p \neq \ell/(\ell+1)$,
it actually follows that, in this case, $w_0$ and $1/p$ are outside the
interval $[(\ell+1)/\ell,\, w^*]$.
Depending on whether $p$ is smaller or larger than $\ell/(\ell+1)$, we
have $w_0 < 1/p$ or $w_0 > 1/p$, respectively.
If $p=\ell(\ell+1)$ then $w^* = (\ell+1)/\ell$ and then $1/p=w_0=w^*$.
Since $f(1) = p^\ell q$, it follows that $w_0 > 1$, in all cases.
Finally, for all sufficiently large $\ell$, we have $p < \ell/(\ell+1)$
and so $1<w_0 < (\ell+1)/\ell$, showing that the limit
of $w_0$, as $\ell \to \infty$, is $1$.
%Let $R_\epsilon:= \max(w_0, 1/p)+\epsilon$.
%Consider the circle $\gamma_\epsilon(t) := R_\epsilon e^{it}$, $0 \le t \le 
%2\pi$.
%Using Rouch\'e's theorem, we can see that, 
%for sufficiently small positive $\epsilon$, the 
%closed curve $f(\gamma_\epsilon(t))$, $0 \le t \le 2\pi$
%surrounds the origin twice. 
%\marginpar{Add}
%\\
%\comments{Lars, do you have a simple proof of this?
%There is a proof, using Rouch\'e's theorem, but it takes
%a few lines.}
%\\
%Hence there are no complex zeros of $f$
%in the domain bounded by the curve $\gamma_\epsilon$, for all sufficiently
%small positive $\epsilon$.
%It follows that all complex zeros $w$ of $f$ satisfy
%$|w| < \max(w_0, 1/p)$.
To show that the only roots $f(w)=0$ with $|w|\le \max(w_0, 1/p)$
are $w_0$ and $1/p$, we need an auxiliary lemma which is probably
well-known but whose proof we supply for completeness:
\begin{lemma}
\label{jlem}
Consider the polynomial $P(z) := c_0+c_1 z+ \cdots + c_n z^n$, $z \in \C$,
with real coefficients such that $c_0 > c_1 > \cdots > c_n > 0$.
Then all the zeros of $P(z)$ lie outside the closed unit ball centered
at the origin. 
\end{lemma}
\begin{proof}
Fix $\lambda > 1$ such that $c_0 > c_1/\lambda >
c_2/\lambda^2 > \cdots > c_n/\lambda^n$ and notice that
\[
c_0 + (z-1) \, P(z/\lambda)
= (c_0-\frac{c_1}{\lambda}) z 
+ (\frac{c_1}{\lambda} -\frac{c_2}{\lambda^2}) z^2 
+ \cdots 
+ (\frac{c_{n-1}}{\lambda^{n-1}} -\frac{c_n}{\lambda^n}) z^n
+ \frac{c_n}{\lambda^n} z^{n+1}.
\]
Therefore, on $|z|=1$,
\[
\big| c_0 + (z-1) \, P(z/\lambda)\big|
\le (c_0-\frac{c_1}{\lambda}) + (\frac{c_1}{\lambda} -\frac{c_2}{\lambda^2})
+ \cdots + (\frac{c_{n-1}}{\lambda^{n-1}} -\frac{c_n}{\lambda^n}) 
+ \frac{c_n}{\lambda^n} = c_0 = |-c_0|.
\]
Rouch\'e's theorem \cite[page 153]{AHL79} implies that 
$(z-1) \, P(z/\lambda)$ and $-c_0$
have the same number of zeros inside the open unit ball
centered at the origin. That is, all zeros of $P(z/\lambda)$
lie outside the open unit ball. Since $\lambda > 1$,
it follows that all zeros of $P(z)$ lie outside the closed unit ball.
\end{proof}

\begin{proof}[End of proof of Proposition \ref{roote}]
Assume first $w_0 p \neq 1$. By polynomial division, write $f(w)$ as
\[
f(w) = -q (1-pw) (w-w_0) S(w),
\]
where
\[
S(w) := \sum_{j=0}^{\ell-1} 
\frac{1-(w_0 p)^{\ell-j}}{1-w_0 p}\, p^j w^j.
\]
%\[
%P(z) = c_0 + c_1 z + \cdots + c_{\ell-1}  z^{\ell-1},
%\]
%and
%\[
%c_j = \frac{1-(w_0 p)^{\ell-j}}{1-w_0 p}\, p^j, \quad
%j=0,1,\ldots,\ell-1.
%\]
%The polynomial $P(z)$ satisfies the assumptions of Lemma \ref{jlem}
%and thus all its zeros lie outside the closed unit circle
%centered at the origin.
Since the positive numbers $(1-(w_0 p)^{\ell-j})/(1-w_0 p)$ decrease with $j$,
it follows from Lemma \ref{jlem} that if $S(w)=0$ then 
$|pw|>1$, i.e., $|w|>1/p$.
On the other hand, since $S(w) =  \sum_{j=0}^{\ell-1}
(1-(w_0 p)^{\ell-j})(w_0p)^j/(1-w_0 p) \, (w/w_0)^j$ and since
the positive numbers $(1-(w_0 p)^{\ell-j})(w_0p)^j/(1-w_0 p)$ decrease with $j$
it follows again by Lemma \ref{jlem} that
$S(w)=0$ implies $|w/w_0|>1$, i.e., $|w| > w_0$.
If $w_0p=1$ then $S(w) = \sum_{j=0}^{\ell-1} (\ell-j) p^j w_j$ and
so $S(w)=0$ implies $|w|>1/p=w_0$.
\end{proof}

We translate this result into an approximation for
the distribution function of $L(n)$.

\begin{proposition}
\label{newappr}
Let $w_0=w_0(\ell)$ be the root of the equation $f(w)=0$ defined in
Proposition \ref{roote}.
If $p\neq \ell/(\ell+1)$ then \footnote{$a(n) \sim b(n)$ stands
for $\lim a(n)/b(n) =1$.} 
\[
\P(L(n)<\ell) \sim  \frac{1-(w_0 p)^\ell}{1-(\ell+1) q (w_0 p)^\ell}
\, w_0^{-n-1},
\text{ as } n \to \infty.
\]
If $p=\ell/(\ell+1)$ then
\[
\P(L(n)<\ell) \sim 2 (\ell/(\ell+1))^{n+1},
\text{ as } n \to \infty.
\]
\end{proposition}
\begin{proof}
Suppose first that $p \neq \ell/(\ell+1)$ and, using
partial fraction expansion, write the expression
\eqref{wtransf} as
\begin{equation}
\label{pfrac}
\frac{g(w)}{f(w)}=
\frac{1-(wp)^\ell}{1-w+(wq)(wp)^\ell}
%= \frac{1-p w_0}{qw_0(\ell+1-\ell w_0)}\, \frac{1}{1-(w/w_0)}
=\frac{c_0}{1-w/w_0} 
+ \frac{h(w)}{j(w)}.
\end{equation}
To do this, we use the fact that $w_0$ is a zero of the denominator
$f(w) = 1-w+(wq)(wp)^\ell$ but not a zero
of the numerator $g(w) = 1-(wp)^\ell$. Also,
$w_0 \neq 1/p$, and both $g(w)$ and $f(w)$ have a zero at $1/p$.
Hence $h(w)$ and $j(w)$ are polynomials with degrees $\ell-2$ and $\ell-1$,
respectively, and $j(w_0) \neq 0$. 
Hence
\[
c_0 = \frac{1}{w_0}\, \lim_{w \to w_0} \frac{(w_0-w) g(w)}{f(w)}
= \frac{g(w_0)}{-f'(w_0)} \, \frac{1}{w_0}
= \frac{1-(w_0 p)^\ell}{1-(\ell+1)q (w_0 p)^\ell}\, \frac{1}{w_0}.
\]
Proposition \ref{roote} tells us that the zeros of $j(w)$ are all
outside the circle with radius $w_0$. Hence,
from expressions \eqref{pfrac} and \eqref{wtransf}, we obtain
\begin{equation}
\label{obtain}
\P(L(n) < \ell) \sim c_0\, w_0^{-n}, \text{ as } n \to \infty,
\end{equation}
which proves the first assertion.
%Note that $h(1/p)=j(1/p)=0$ and that all
%zeros of $j(w)$ except $1/p$ lie outside the circle with radius $\max(
%w_0, 1/p)$.
%From this observation and \eqref{wtransf}, the first assertion follows.
%we conclude that, for all $\ell \in \N$,
%\[
%F_\ell(n) \sim \frac{1-p w_0}{qw_0(\ell+1-\ell w_0)} \, w_0^{-n},
%\text{ as } n \to \infty
%\]
%(in the sense that the ratio of the two tends to $1$ as $n \to \infty$).
If $p=\ell/(\ell+1)$ 
%corresponds, again from the result of Proposition \ref{roote}, 
%to the case where the numerator $g(w)$
%has a simple zero at $1/p$ and the denominator $f(w)$
%a double zero at $w_0=1/p$.
then $w=1/p=w_0$ is a simple zero for $g(w)$ and a double zero for $f(w)$.
%Taking this into account and the partial fraction expansion 
%\eqref{pfrac}, we find
Hence \eqref{pfrac} gives
\[
c_0 = - \frac{g'(w_0)}{\frac{1}{2} f''(w_0)}\, \frac{1}{w_0}
=
\frac{-p^\ell q w_0^{\ell-1}}
{\frac{1}{2} p^\ell q \ell(\ell+1) w_0^{\ell-1}}\,\frac{1}{w_0}
= \frac{2}{q(\ell+1) w_0} 
= \frac{2\ell}{\ell+1}.
\]
%But $g'(w) = - p^\ell q w^{\ell-1}$, 
%$f''(w) = p^\ell q \ell(\ell+1) w^{\ell-1}$,
%so
%\[
%c_0 = \frac{2}{q(\ell+1)} \, \frac{1}{w_0}= \frac{2}{(1-p)(\ell+1)} = 
%\frac{2}{w_0}.
%\]
%Therefore,
%\[
%\P(L(n) < \ell) \sim \frac{2}{w_0}\, w_0^{-n}
%= 2 w_0^{-n-1} = 2 (\ell/(\ell+1))^{-n-1},
%\]
Since the zeros of $j(w)$ are outside the circle with radius
$w_0=(\ell+1)/\ell$, \eqref{obtain} holds. This proves
the second assertion.
%%\[
%%\frac{1-(wp)^\ell}{1-w+(wq)(wp)^\ell}
%%=  \frac{2\ell}{\ell+1}\, \frac{1}{1-\ell w/(\ell+1)}
%%+ \frac{h(w)}{j(w)},
%%\]
%%which implies the second assertion.
%%\[
%%F_\ell(n) \sim 2 (\ell/(\ell+1))^{n+1},
%%\text{ as } n \to \infty.
%%\]
%proving the second assertion.
\end{proof}
Since these approximations are valid for all $\ell$, they nicely complement
the Poisson approximation discussed earlier.
For $n, \ell \to \infty$, such that $n p^\ell \asymp 1$, we have
%\[
%w_0(\ell) = 1+p^\ell q + O(\ell/n^2).
%\]
$w_0(\ell) = 1+p^\ell q + O(\ell/n^2)$.
From the approximation above, we find
%\[
%\P(L_\ell(n) < \ell) \approx e^{-n p^\ell q},
%\]
$\P(L_\ell(n) < \ell) \approx e^{-n p^\ell q}$
which is asymptotically equivalent to the Poisson approximation.

\subsection{Numerical comparisons of the two approximations}
\label{Num}
We numerically compute $\P(L(n) \ge \ell)$, first using the
exact formula \eqref{Fellexact}, then using the Poisson
approximation \eqref{aapp}, and finally using the approximation
suggested by Proposition \ref{newappr}. We see, as expected,
that for small values of $\ell$ compared to $n$, the second approximation
outperforms the first one.

First, we let $\ell=2$.
Then
%\[
%f(w) = 1-w+p^2 q w^3 = (pw-1) (pq w^2 + qw -1),
%\]
$f(w) = 1-w+p^2 q w^3 = (pw-1) (pq w^2 + qw -1)$,
and so
\[
w_0 = \frac{\sqrt{1+4p/q}-1}{2p}.
\]
The approximation suggested by Proposition \ref{newappr} is
%\[
%\P(L(n) < 2) \sim \frac{1-(w_0 p)^2}{1-3q(w_0 p)^2}\, w_0^{-n+1},
%\]
%assuming that $p \neq 2/3$, whereas, for $p=2/3$, we have
%\[
%\P(L(n) < 2) \sim 2 (2/3)^{-(n+1)}.
%\]
\[
\P(L(n) < 2) \sim
\begin{cases}
\displaystyle 
\frac{1-(w_0 p)^2}{1-3q(w_0 p)^2}\, w_0^{-n+1}, & \text{ if } p \neq 2/3,
\\
2 (2/3)^{-(n+1)}, & \text{ if } p = 2/3
\end{cases}
\]
%Let $p=1/2$, $1/3$ and $4/5$.
Below are some numerical values.
%%%%%%%%%%%%%%%%
\begin{center}
\begin{tabular}{l|l|l|l}
   \multicolumn{4}{c}{$\P(L(n) \ge \ell)$ for $p=1/2$ and $\ell=2$} \\
$n$ & exact & Poisson approx. (error) & second approx. (error)
\\\hline
5 & 0.59375 & 0.46474 (22\%) & 0.59426 (0.086\%) \\
7 & 0.73438 & 0.58314 (20.6\%) & 0.73445 (0.01\%) \\  
10 & 0.85938 & 0.71350 (17\%) & 0.8594 (0.002\%) \\
20 & 0.98311 & 0.91792 (6.63\%) & 0.98312 (0.0010\%)
\end{tabular}
\end{center}
%%%%%%%%%%%%%%%%
%%%%%%%%%%%%%%%%
\begin{center}
\begin{tabular}{l|l|l|l}
   \multicolumn{4}{c}{$\P(L(n) \ge \ell)$ for $p=1/3$ and $\ell=2$} \\
$n$ & exact & Poisson approx. (error) & second approx. (error)
\\\hline
5 & 0.32510 & 0.28347 (12.8\%) & 0.32557 (0.14\%) \\
7 & 0.44033  & 0.38213 (13.2\%) & 0.44080 (0.11\%) \\ 
10 & 0.57730 & 0.50525 (12.5\%) & 0.57779 (0.08\%) \\
20 & 0.83415 & 0.76411 (8.4\%) & 0.83453 (0.05\%) 
\end{tabular}
\end{center}
%%%%%%%%%%%%%%%%
%%%%%%%%%%%%%%%%
\begin{center}
\begin{tabular}{l|l|l|l}
   \multicolumn{4}{c}{$\P(L(n) \ge \ell)$ for $p=4/5$ and $\ell=2$} \\
$n$ & exact & Poisson approx. (error) & second approx. (error)
\\\hline
5 & 0.94208 & 0.64084 (32.0\%) & 0.94386 (0.189\%) \\
7 & 0.98509 & 0.72196 (26.71\%) & 0.98526 (0.0173\%) \\
10 & 0.9980232 & 0.8106201 (18.78\%) & 0.9980179 (0.00052\%) \\
20 & 0.9999975 & 0.9473453 (5.265\%) & 0.9999975 ($0.000003$\%) 
\end{tabular}
\end{center}
%%%%%%%%%%%%%%%%

Next, we increase the value of $\ell$
and pick two different values for $p$.
We solve, in each case, the equation $f(w)=0$ numerically.
%%%%%%%%%%%%%%%%
\begin{center}
\begin{tabular}{l|l|l|l}
  \multicolumn{4}{c}{$\P(L(n) \ge \ell)$ for $p=1/2$ and $\ell=7$}  \\
$n$ & exact & Poisson approx. (error) & second approx. (error)
\\\hline
100 & 0.31752 & 0.31002 (2.36\%) & 0.19644 (38.13\%) \\
500 & 0.86364 & 0.85537 (0.96\%) & 0.8372 (3.06\%) \\
1500 & 0.99757 & 0.99709 (0.048\%) & 0.99700 (0/057\%) \\
3000 & 0.9999941986 & 0.9999916997 (0.00025\%) & 0.9999931928 (0.00010\%)
\end{tabular}
\end{center}
%%%%%%%%%%%%%%%%
%%%%%%%%%%%%%%%%
\begin{center}
\begin{tabular}{l|l|l|l}
   \multicolumn{4}{c}{$\P(L(n) \ge \ell)$ for $p=2/3$ and $\ell=10$}  \\
$n$ & exact & Poisson approx. (error) & second approx. (error)
\\\hline
100 & 0.43531 & 0.41583 (4.475\%) & 0.46433 (6.667\%) \\
500 & 0.95209 & 0.94214 (1.045\%) & 0.95480 (0.285\%) \\
1500 & 0.999900 & 0.999821 (0.00790\%) & 0.999905 (0.00050\%) \\
3000 & 0.9999999904 & 0.9999999694 (2.1$\times 10^{-6}$\%)
& 0.9999999908 (0.04 $\times 10^{-6}$\%)
\end{tabular}
\end{center}
%%%%%%%%%%%%%%%%

\section{Discussion and open problems}
Gordon, Schilling and Waterman \cite{GSW86}
developed an extreme value theory for long runs. As mentioned therein,
it is intriguing that the longest run possesses no limit distribution,
and this is based on an older paper by Guibas and Odlyzko \cite{GO80}.

We have not touched upon the issue of more general processes generating
heads and tails. For example, Markovian processes. The portmanteau identity
can be generalized to include the Markovian dependence and this can
be the subject for future work provided that a suitable motivation
be found.

Another set of natural questions arising is to what extent we
have weak approximation of $R(n)$ on a function space (convergence to
a Brownian bridge?),
as well as the quality of such an approximation.

\subsection*{Acknowledgments}
We would like to thank Joe Higgins for the short proof of Lemma \ref{jlem}.

\small
\vspace*{1cm}

\noindent
\begin{minipage}[t]{9cm}
\small \sc
Lars Holst\\
Department of Mathematics\\
Royal Institute of Technology\\
SE-10044 Stockholm\\
Sweden\\
Email: {\tt lholst@kth.se}
\end{minipage}
\begin{minipage}[t]{6cm}
\small \sc
Takis Konstantopoulos\\
Department of Mathematics\\
Uppsala University\\
SE-75106 Uppsala\\
Sweden\\
E-mail: {\tt takis@math.uu.se}
\end{minipage}

\end{document}